\documentclass[11pt]{amsart} 
\pdfoutput=1
\usepackage{amsfonts}
\usepackage{amssymb}
\usepackage{amsmath}
\usepackage{amsthm}
\usepackage{enumerate}
\usepackage{graphicx}
\usepackage[backend=bibtex, style=nature]{biblatex}

\usepackage[pdftex,pdfpagelabels,bookmarks,hyperindex,hyperfigures,
            pdfauthor={Thomas Gunnison Brooks},
            pdftitle={Non-negative Curvature and Conullity of the Curvature Tensor},
            pdfsubject={Riemannian geometry}]{hyperref}

\bibliography{conullity}

\author{Thomas Brooks}

\title{Non-negative Curvature and Conullity of the Curvature Tensor}

\newcommand{\R}[0]{\mathbb{R}}

\newcommand{\N}[0]{\mathbb{N}}

\newcommand{\brak}[1]{\left [ #1 \right ]}
\newcommand{\paren}[1]{\left ( #1 \right )}
\newcommand{\curly}[1]{\left \{ #1 \right \}}

\newcommand{\chevron}[1]{\left \langle #1 \right \rangle}

\newcommand{\pder}[2]{\tfrac{\partial#1}{\partial#2}}

\DeclareMathOperator{\Sec}{sec}
\DeclareMathOperator{\Scal}{Scal}

\DeclareMathOperator{\trace}{tr}

\newtheorem{theorem}[equation]{Theorem}
\newtheorem*{theorem*}{Theorem}
\newtheorem{prop}[equation]{Proposition}
\newtheorem*{prop*}{Proposition}
\newtheorem{lemma}[equation]{Lemma}
\newtheorem*{lemma*}{Lemma}

\newtheorem*{corollary*}{Corollary}

\newtheorem*{question*}{Question}

\newtheorem*{problem*}{Problem}

\newtheorem*{conjecture*}{Conjecture}

\theoremstyle{remark}

\theoremstyle{definition}

\begin{document}

\begin{abstract}
The conullity of a curvature tensor is the codimension of its kernel.
We consider the cases of conullity two in any dimension and conullity three in dimension four.
We show that these conditions are compatible with non-negative sectional curvature only if either the manifold is diffeomorphic to $\R^n$ or the universal cover is an isometric product with a Euclidean factor.
Moreover, we show that finite volume manifolds with conullity 3 are locally products.
\end{abstract}

\maketitle

Let $(M^n,g)$ be a Riemannian manifold with curvature tensor $R$.
Define the distribution
\[ \ker R_p := \curly{ X \in T_p M: R(X,Y)Z = 0 \mbox{ for all } Y,Z \in T_p M}. \]
We say that $M^n$ has \emph{nullity k} if at every point $p \in M$, $\ker R_p$ has dimension $k$.
We will study manifolds with conullity 2 or 3.
The simplest example is $M = \Sigma^2 \times \R^{n-2}$ with the product metric and $\Sigma^2$ any surface.
This manifold has conullity 2 if $\Sigma^2$ has nowhere zero Gaussian curvature.
There are many other examples with conullity two which are locally irreducible, see \cite{CN2_book} and refrences therein.

Our two main results concern such manifolds under the assumption of non-negative sectional curvature.
\begin{theorem}
\label{thm:pos_scal_CN2}
Suppose that $M^n$, $n \geq 2$, is complete, has conullity 2 and $\sec \geq 0$.
If its universal cover is irreducible, then $M^n$ is diffeomorphic to $\R^n$.
\end{theorem}
\begin{theorem}
\label{thm:pos_curv_CN3}
Suppose that $M$ is a complete 4-dimensional Riemannian manifold that has nullity one and $\Sec \geq 0$.
If its universal cover is irreducible, then $M$ is diffeomorphic to $\R^4$.
\end{theorem}

Additionally, we prove the following, without any curvature assumption.
\begin{theorem}
\label{thm:fin_volume_CN3}
Assume that $M^4$ is a complete, finite volume Riemannian manifold with positive nullity.
If $M$ has $\dim \ker R = n-3$ everywhere, then the universal cover of $M$ splits  isometrically as $D \times \R$ for some 3-manifold $D$.
\end{theorem}

See also Theorem~\ref{thm:local_fin_volume_CN3} for a local version of this result.

In \cite{olmos}, the authors found a homogenous (and hence complete) example with conullity 3 and $\Scal < 0$.
We give an example with $\sec \geq 0$ but we do not know of any which are complete.

In Section~\ref{section:preliminaries}, we give basic definitions and properties of manifolds of positive nullity.
In Section~\ref{section:conullity2}, we prove Theorem~\ref{thm:pos_scal_CN2}, and in Section~\ref{section:conullity3}, we prove Theorems~\ref{thm:pos_curv_CN3} and~\ref{thm:fin_volume_CN3}.
In Section~\ref{section:example}, we give an example of a locally irreducible conullity 3 metric on $\R^4$.

The results in this paper are part of the author's Ph.D. thesis \cite{thesis} under the direction of Dr. Wolfgang Ziller.
The author is deeply grateful to Dr. Ziller for his invaluable guidance throughout the development and writing of these results.

\section{Preliminaries}
\label{section:preliminaries}

It is well known that $\ker R$ has complete, totally geodesic leaves on the open subset where $\dim \ker R$ is minimal \cite{maltz}.
Moreover, these leaves are flat their tangent space is in $\ker R$.
Any geodesic contained in a leaf of $\ker R$ is called a \emph{nullity geodesic}, and all geodesics starting at $p \in M$ with tangent vector $T \in \ker R$ are nullity geodesics.
Since $\ker R$ has totally geodesic leaves, the orthogonal distribution $\ker R^\perp$ is parallel along the leaves of $\ker R$.

Following the conventions of \cite{graph_manifolds, dajczer_gromoll_C}, define the \emph{splitting tensor} $C_T$ for any $T \in \ker R$ by
\begin{equation}
(C_T)_p(X) = - (\nabla_X T)^{\ker R_p^\perp}
\end{equation}
where $(\cdot)^{\ker R_p^\perp}$ denotes the orthogonal projection onto $(\ker R_p)^\perp$.
Notice that if $C_T \equiv 0$ for all $T$, then the metric splits locally.

Moreover, from \cite{graph_manifolds}, for vector fields $U,S \in \ker R$,
\begin{align*}
C_{\nabla_U S} X
    &= -(\nabla_U \nabla_X S)^{\ker R^\perp} - (\nabla_{[X,U]} S)^{\ker R^\perp} \\
    &= (\nabla_U C_S) X + C_S(\nabla_U X) - C_S([U,X]^{\ker R^\perp}) \\
    &= (\nabla_U C_S) X + C_S(\nabla_X U) \\
    &= (\nabla_U C_S) X - C_S C_U X
\end{align*}
Hence, we obtain a Ricatti-type equation,
\begin{equation}
\label{eqn:super_ricatti}
\nabla_U C_S = C_{\nabla_U S} + C_S C_U.
\end{equation}

Along a nullity geodesic $\gamma(t)$ with tangent vector $T \in \ker R$, we can choose a parallel basis $\curly{e_1, \ldots, e_k}$ of $\ker R^\perp$.
Then $C_T$ written in this basis is a matrix $C(t)$ along $\gamma(t)$ satisfying
\begin{equation}
\label{eqn:ricatti}
C'(t) = C^2
\end{equation}
and hence has solutions $C(t) = C_0(I - tC_0)^{-1}$ for some matrix $C_0 = C(0)$.
Hence all real eigenvalues of $C_T$ must be zero.

When $M$ has conullity at most $2$, then $C_T$ is a $2 \times 2$ matrix and hence either is nilpotent or has two non-zero complex eigenvalues.
When $M$ has conullity at most $3$, then $C_T$ is $3 \times 3$ matrix and hence always has a zero eigenvalues.
Moreover, $C_T$ is again either nilpotent or has two non-zero complex eigenvalues.
These two cases lead to qualitatively different behavior.

We will make use of the following de Rham-type splitting result, see \cite{graph_manifolds}.
\begin{prop}
\label{prop:de_rham_splitting}
Let $M$ be a complete Riemannian manifold, and $V \subset M$ a connected open subset on which the parallel rank $k$ distribution $\ker R$ has leaves that are complete.
Then, the universal cover of $V$ is isometric to $\tilde D \times \R^k$, where $\tilde D$ is the universal cover of a maximal leaf $D$ of $\ker R^\perp$.
Furthermore, the normal exponential map $\exp^\perp: T^\perp D \rightarrow V$ is an isometric covering map if $T^\perp D$ is equipped with the induced connection metric.

\end{prop}

\section{Conullity 2}
\label{section:conullity2}

We now assume throughout this section that, for $n \geq 3$, $M^n$ has conullity exactly 2 and $\sec \geq 0$.
We work towards the proof of Theorem~\ref{thm:pos_scal_CN2}.
Since $\sec \geq 0$, $M$ has a soul $S \subset M$, see  \cite{soul_thm, soul_conj}.

The following proposition from \cite{graph_manifolds} covers the finite-volume case without a curvature assumption.
\begin{prop}
\label{prop:finite_volume_splitting}
If a complete manifold $M$ has conullity at most 2 and has finite volume, then its universal cover $\widetilde M$ splits isometrically as $\Sigma \times \R^{n-2}$ for some complete surface $\Sigma$.
\end{prop}
We will use this result by applying it to a soul of $M$ in the case where $\ker R$ is orthogonal to $S$.

The following lemma will apply for the opposite case, where $TS \subset \ker R$ and will also be used in the proof of Theorem~\ref{thm:pos_curv_CN3}.
\begin{lemma}
Suppose that $M$ has a soul $S$ of dimension at least one.
If $S$ is flat, $\widetilde M$ splits isometrically with a Euclidean factor.
\label{lemma:flat_soul}
\end{lemma}

\begin{proof}
In this case, since $S$ is flat, we know that its universal cover $\widetilde S$ is flat $\R^m$.

Let $\nu(S)$ be the normal bundle of $S$ in $M$.
If $\pi: \widetilde S \rightarrow S$ is the universal covering of $S$, then $\widetilde M$ is diffeomorphic to the pullback bundle
\[ \pi^*(\nu(S)) = \curly{(\widetilde p, X) \Big| \; \widetilde p \in \widetilde S, \; X \in T_{\pi(\widetilde p)} S^\perp}. \]
This follows from the covering map $\pi^*(\nu(S)) \rightarrow \nu(S) \approx M$ to the normal bundle of $S$, which is diffeomorphic to $M$.
Specifically, this map is $(\widetilde p, X) \mapsto (\pi(\widetilde p), X)$.
Moreover $\pi^*(\nu(S))$ is simply connected since $\pi_1(\widetilde S) = 0$.
Hence $\pi^*(\nu(S))$ is the universal cover of $M$.
So $\widetilde M$ is diffeomorphic to $\pi^*(\nu(S))$, a vector bundle over Euclidean space $S = \R^m$.
Hence $\widetilde M$ is diffeomorphic to $\R^n$.

Suppose that $m > 0$ so that the soul $S$ is not just a point.
The fact that $\pi_1(S) = \pi_1(M)$ implies that $\widetilde S$ embeds in $\widetilde M$, since distinct homotopy classes of paths in $S$ are still distinct in $M$.
Since $S$ is totally geodesic and totally convex in $M$, so is $\widetilde S$ in $\widetilde M$.

Now take a line $L$ in $\widetilde S = \R^m$ and any two points $x,y$ on the line $L$.
Then any minimizing geodesic in $M$ from $x$ to $y$ must lie in $\widetilde S$, since $\widetilde S$ is totally convex, and the only such geodesic is the line $L$ itself.
Hence, by the splitting theorem, $\widetilde M$ splits isometrically as $N^{n-1} \times \R$ \cite{toponogov}.
Here $N^{n-1}$ has a soul with dimension at most $m-1$.
This process can be repeated until $\widetilde M = N^{n-m} \times \R^m$ isometrically with flat $\R^m$ for some manifold $N^{n-m}$ with soul a point
In particular, $N^{n-m}$ is diffeomorphic to $\R^{n-m}$.
\end{proof}

Since $M$ has conullity 2, at each point $p \in M$ there is an orthonormal basis of the form $\{e_1, e_2, T_1, \ldots, T_{n-2}\}$ of $T_p M$ with $T_i \in \ker R$ and $\Sec(e_1,e_2) = \Scal$.
Now we consider how $T_1, ... T_{n-2}$ relates to the soul of $M$.

\begin{lemma}
\label{lemma:soul_projection}
If $T \in \ker R_p$ at a point $p \in S$, then the orthogonal projections $T^S \in T_p S$ and $T^N \in T_p S^\perp$ are also in $\ker R_p$.
\end{lemma}

\begin{proof}
First observe that since $T = T^S + T^N \in \ker R_p$, that
\begin{align}
\label{eqn:swap}
\chevron{R(T^S, X) Y, Z} = -\chevron{R(T^N, X)Y, Z}
\end{align}
for any $X,Y,Z$.
Take a unit vector $e \in T_p M$ orthogonal to $T$ and write $e^S$ and $e^N$ as its projections.
Then
\begin{align*}
\Sec(e,T^N)
    &= \underbrace{\chevron{R(e^N, T^N) T^N, e^N}}_{(a)}
        + \underbrace{2\chevron{R(e^S,T^N) T^N, e^N}}_{(b)}
        + \underbrace{\chevron{R(e^S, T^N) T^N, e^S}}_{(c)}
\end{align*}
This last term $(c)$ is 0 since it is the sectional curvature of one of the flat strips from the proof of the Soul Conjecture \cite{soul_conj}.
The first term $(a)$ can be written using~\eqref{eqn:swap} as
\[ (a) = - \chevron{R(e^N, T^S) T^N, e^N} = \chevron{R(e^N, T^S) T^S, e^N}\]
which is again the curvature of a flat strip and hence zero.

For $(b)$, we use the fact that the flat strip spanned by $e^S$ and $T^N$ is totally geodesic, and so $R(e^S, T^N) T^N$ is in the span of $\{e^S, T^N\}$ and hence $(b) = 0$.

This shows that $T^N$ has $\Sec(T^N, \cdot) = 0$.
Using~\ref{eqn:swap} twice then also gives that $\Sec(T^S, \cdot) = 0$.

	This is sufficient to show that $T^N$ and $T^S$ are in $\ker R_p$, as any $X \not\in \ker R_p$ has $\Sec(X,Y) \not= 0$ for some $Y$.

\end{proof}

Our next lemma tells us how to choose a basis of the tangent space at a point of the soul that fits nicely with both the soul structure and the conullity 2 structure.
The result is illustrated in the case of four dimensional manifolds in Figure~\ref{fig:soul_basis}.

\begin{lemma}
\label{lemma:soul_basis}
For $p \in S$, there exists an orthonormal basis $B = \{e_1, e_2, T_1, \ldots, T_{n-2}\}$ of $T_p M$
so that each basis vector $v \in B$ is either in $T_p S$ or in $T_p S^{\perp} \subset T_p M$ and $B$ satisfies the relations
\[ R(T_j, \cdot) = 0, \quad \Sec(e_1, e_2) = \Scal. \]
Moreover, $e_1$ and $e_2$ are either both in $T_p S$ or both in $T_p S^\perp$.
\end{lemma}

\begin{proof}
Pick any basis $S_1, \ldots, S_{n-2}$ of $\ker R_p$.
Then $S_1^N, \ldots, S_{n-2}^N, S_1^S, \ldots, S_{n-2}^S$ also spans $\ker R_p$ by Lemma~\ref{lemma:soul_projection}, so take a subset which is a basis and call it $T_1, \ldots, T_{n-2}$.
Now chose $e_1, e_2$ perpendicular to the span of $T_1, \ldots, T_{n-2}$ with each $e_i$ either in $T_p S$ or $T_p S^\perp$.
Then $e_1, e_2$ span $\ker R_p^\perp$ and $\{e_1, e_2, T_1, \ldots, T_{n-2}\}$ is our desired basis.

Moreover, note that if $e_1 \in T_p S$ and $e_2 \in T_p S^\perp$, then there is a flat strip spanned by $e_1$ and $e_2$, so $\Sec(e_1, e_2) = 0$, which is a contradiction with the assumption that $\Scal > 0$ everywhere.
So $e_1$ and $e_2$ must both be in $T_p S$ or both be in $T_p S^\perp$.
\end{proof}

We now prove Theorem~\ref{thm:pos_scal_CN2}.
\begin{proof}

\begin{figure}
	\includegraphics[scale=0.2]{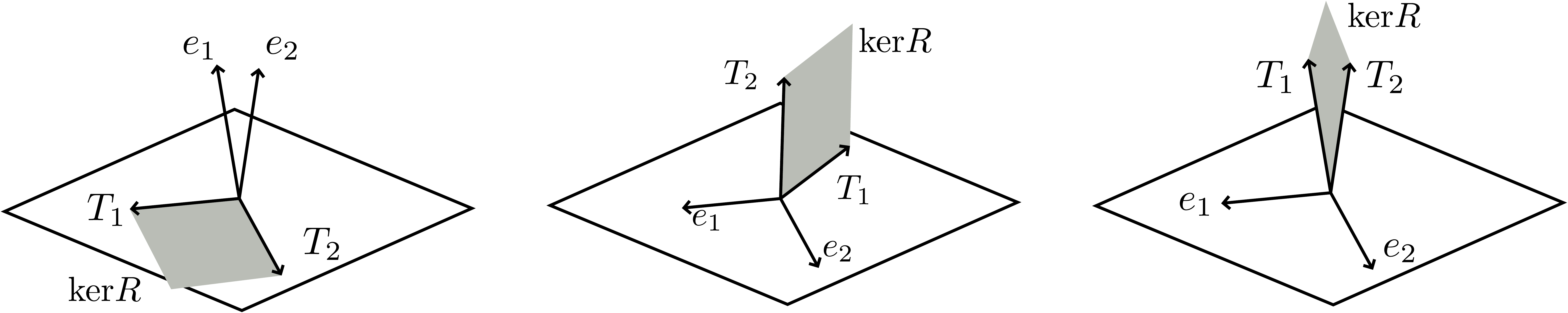}
	\caption{All possibilities for the orthonormal basis from Lemma~\ref{lemma:soul_basis} illustrated in the case of a four dimensional manifold and two or three dimensional soul.}
	\label{fig:soul_basis}.
	\centering
\end{figure}

Let $S$ be a soul of $M$.
If $S$ has dimension zero, then $M \approx \R^n$ and we are done.
Since Proposition~\ref{prop:de_rham_splitting} covers the case where $M$ is compact, we may also assume that $S$ does not have dimension $n$.

Using Lemma~\ref{lemma:soul_basis}, if $e_1$ and $e_2$ are in $T_p S$ at one point of $S$, they must be so at every point of $S$.
So there are now two cases: the case where $e_1, e_2 \in T_p S^\perp$ for all $p \in S$ and the case where $e_1, e_2 \in T_p S$ for all $p \in S$.

In the first case, the soul of $M$ is flat.
We apply Lemma~\ref{lemma:flat_soul} to conclude that $\widetilde M$ splits isometrically.

In the second case, the soul $S$ of $M$ is a compact manifold with conullity 2 at each point.
So we may apply Proposition~\ref{prop:finite_volume_splitting} to the soul to get that $\widetilde S$ is isometric to $\widetilde D \times \R^{m-2}$ where $m$ is the dimension of $S$.
Here $\widetilde D$ is a simply connected surface with positive Guassian curvature.
The curvature on $\widetilde D$ is bounded away from zero since $S$ is compact, and hence $\widetilde D$ is compact and therefore diffeomorphic to $\mathbb{S}^2$.

Now we examine the splitting tensor of $M$ at $p \in S$.
If $T \in T_p S$ and $T \in \ker R$, then $C_T = 0$ by the splitting of $\widetilde S$.
Otherwise, assume that $T$ is perpendicular to $T_p S$.
For $X \in \ker R_p^\perp$, the flat strip spanned by $X$ and $T$ is totally geodesic.
Since $C_T$ is a tensor, we can choose to consider extensions of $X$ and $T$ to vector fields contained in that flat strip.
For these extensions, $\nabla_X T$ is in the span of $X$ and $T$.
Since $C_T(X) \in \ker R_p^\perp$, it must be perpendicular to $T$ and hence $X$ is an eigenvector of $C_T(X)$ with a real eigenvalue.
The only possible such eigenvalue is $0$.
So $C_T = 0$ as well.

So all splitting tensors are zero on $S$.
For any other point $p \in M$, $p = \exp_{p_0}(U)$ for some $p_0 \in S$ and $U \in T_pS^\perp$.
Since $e_1, e_2 \in T_p S$, we know that $U \in \ker R_p$.
By \eqref{eqn:super_ricatti}, $C_U \equiv 0$ along $\gamma(t) = \exp_{p_0}(tU)$ since $C_U = 0$ at $p_0 \in S$.
For any $T \in \ker R_{p_0}$, we know that $C_T = 0$ at $p_0$.
By \eqref{eqn:super_ricatti} extending $T$ parallel along $\gamma$, we get that
\[ \nabla_U C_T = C_T C_U = 0. \]
Hence, $C_T \equiv 0$ along $\gamma$ and in particular $C_T = 0$ at $p$.
Since $\ker R$ is parallel along $\gamma$, $C_T = 0$ at $p$ for all $T \in \ker R_p$.

So all splitting tensors are identically zero on $M$.
By Proposition~\ref{prop:de_rham_splitting}, we conclude that $\widetilde M$ splits isometrically as $\widetilde D \times \R^{n-2}$ with the Euclidean metric on $\R^{n-2}$ for some surface $\widetilde D$.
\end{proof}


\section{Conullity 3}
\label{section:conullity3}

We will first prove Theorem~\ref{thm:fin_volume_CN3} and then use it to prove Theorem~\ref{thm:pos_curv_CN3}.

Recall that in conullity at most 3, any splitting tensor $C_T$ is a $3 \times 3$ matrix in a parallel basis along $\gamma$.
Hence $C_T$ has at least one real eigenvalue.
Recall that the real eigenvalues of $C$ are all zero, and hence $C_T$ has 0 as an eigenvalue.
The two possibilities are then that either $C_T$ has two complex eigenvalues and one 0 eigenvalue or that $C_T$ is nilpotent.

\subsection{Finite Volume}
\label{sec:finvolume_split_proof}

We now prove a more general version of Theorem~\ref{thm:fin_volume_CN3}, following closely the strategy in \cite{graph_manifolds} for the proof of Proposition~\ref{prop:finite_volume_splitting}.
\begin{theorem}
\label{thm:local_fin_volume_CN3}
Assume that $M^4$ is a complete, finite volume Riemannian manifold with positive nullity.
Let $V$ be a connected open subset of $M$ on which the nullity leaves are complete and $\dim \ker R = 1$.
Then the universal cover of $V$ splits isometrically as $\tilde D \times \R$ where $D$ is a maximal leaf of $\ker R$ in $V$.
\end{theorem}
\begin{proof}
Define $C := C_T$ on $V$.
We will show that $C = 0$ and hence Proposition~\ref{prop:de_rham_splitting} finishes the proof.
Fix a nullity geodesic $\gamma$ and define $C(t)$ to be $C$ at the point $\gamma(t)$.

First, we look at the case where $C$ has two non-zero complex eigenvalues and one zero eigenvalue.
Then in an appropriate choice of parallel basis along $\gamma$,
\begin{equation}
C(t) = \begin{pmatrix} A(t) & 0 \\ 0 & 0 \end{pmatrix}
\end{equation}
where $A$ is a $2 \times 2$ matrix with 2 complex eigenvalues.
The differential equation \eqref{eqn:ricatti} then easily implies that
\begin{equation} \trace A(t) = \frac {\trace A_0 - 2 t \det A_0}{1 - t \trace A_0 + t^2 \det A_0}
\mbox{ and }
\det A(t) = \frac {det A_0}{1 - t \trace A_0 + t^2 \det A_0}
\label{tr_det_solns}
\end{equation}

Take $B \subset V$ a small compact neighborhood.
Since $\det A(0) > 0$, there is some time $t_0$ so that $\trace C(t) = \trace A(t) \leq 0$ for all $t \geq t_0$.
In the second case, where $C$ is nilpotent, then $\trace C = 0$.
In either case, $\trace C \geq 0$ for $t \geq t_0$ for some $t_0$.

Note that
\begin{equation}
\mbox{div} \; T = \trace \nabla T = - \trace C.
\end{equation}
Now define $B_t := \phi_{t + t_0}(B)$ where $\phi$ is the flow along $T$.
Then
\begin{equation}
\frac{d}{dt} \mbox{vol} \; B_t= \int_B \frac {d}{dt} \phi_{t+t_0}^*(d\,vol)
    = \int \mbox{div\,} T
    = - \int_B \trace C(t + t_0)
    \geq 0
\end{equation}
for all $t \geq 0$.
Hence, the flow of $T$ is volume non-decreasing and we get, by weak recurrence, a sequence of compact neighborhoods $B_{n_k}$, with $\{n_k\} \in \N$ an increasing sequence, so that $B_{n_k}~\cap~B_{n_0}~\not=~\emptyset$.
This gives a sequence of points, $p_k := \phi_{t_0 + n_k}(q_k) \in B_{n_k} \cap B_{n_0}$, with $q_k \in B$, with an accumulation point $p \in B_{n_0} \subset V$.

First consider $V' \subset V$, the open subset on which $C$ has non-zero complex eigenvalues.
By~\eqref{eqn:super_ricatti}, $V'$ is invariant under the flow $\phi_t$ of $T$.
The sequence of points $p_k \rightarrow p$ and \eqref{tr_det_solns} give
\begin{equation} \det A_{T(p)} = \lim_{k \rightarrow \infty} \det A_{T(p_k)}
    = \lim_{k\rightarrow \infty} \frac {\det A_{T(q_k)} }{ 1 - (t_0 - n_k) \trace A_{T(q_k)} + (t_0 - n_k)^2 \det A_{T(q_k)}}
    = 0
\end{equation}
where again $A_{T(q_k)}$ is the $2 \times 2$ block of $C_{T(q_k)}$ with two non-zero complex eigenvalues.
Therefore $A_{T(p)} = 0$ and so $C = 0$.

For the above, note that $\trace A$ and $\det A$ are both independent of the choice of coordinates.
Indeed, $\trace A = \trace C$ and if $\lambda_1, \lambda_2, \lambda_3$ are the eigenvalues of $C$, then
\begin{equation}
\det A_{T(q_k)} = \lambda_1 \lambda_2 =  \lambda_1 \lambda_2 + \lambda_1 \lambda_3 + \lambda_2 \lambda_3  = \frac{(\trace C)^2 - \trace (C^2)} 2.
\end{equation}

Consider next the other case and define $V^*$ to be the open subset of $V$ on which $C$ is nilpotent and non-zero.
The previous case differs only slightly from the argument in conullity 2, but the nilpotent case requires significantly more computations than in the case of conullity 2.

First, we find vector fields on $V^*$ giving a canonical orthonormal basis.
Observe that
\begin{equation}\dim \ker C^2 = \dim \ker C + 1
\end{equation}
since $C$ is $3 \times 3$, nilpotent, and non-zero.
Define $e_2$ to be a unit vector field spanning $\ker C^2 \cap (\ker C)^\perp$ on $V^*$, passing to a double cover of $V$ if necessary.
Then let $e_1$ be a unit vector field parallel to $C(e_2)$ and $e_3$ a unit vector field perpendicular to $e_1$ and $e_2$, passing to a cover of $V^*$ if necessary.
This gives an orthonormal basis of vector fields $\{e_1, e_2, e_3\}$ on which we can write $C$ as
\begin{equation}C = 
\begin{pmatrix}
0 & a & c \\
0 & 0 & b \\
0 & 0 & 0
\end{pmatrix}
\end{equation}
Note that by this construction, $a$ is non-zero at every point on $V^*$, though $b$ and $c$ possibly could be zero.
Moreover, \eqref{eqn:ricatti} shows that $\ker C$ and $\ker C^2$ are parallel along nullity geodesics, and hence $e_1, e_2, e_3$ are as well.

Then~\eqref{eqn:ricatti} gives
\begin{equation}C(t) = 
\begin{pmatrix}
0 & a & c+tab \\
0 & 0 & b \\
0 & 0 & 0
\end{pmatrix}
\end{equation}
where $a,b,c$ are independent of $t$.

Since the flow $\phi_t$ of $T$ is volume preserving ($\trace C = 0$), the Poincar\'e recurrence theorem says that for almost all $p \in V^*$, there exists a sequence $t_n \rightarrow \infty$ with $\phi_{t_n}(p) \rightarrow p$.
Hence $\chevron{C(t) e_3, e_1} = c+ tab$ must be constant, not linear, and hence $b = 0$ since $M$ has finite volume.
Thus $\ker C(t)$ is 2 dimensional.
This allows us to choose a better basis (again, in a cover of $V^*$, if necessary).
Let $e_2$ be perpendicular to $\ker C$, $e_1$ parallel to $C e_2$ and $e_3$ perpendicular to $e_1, e_2$.
Then $\{e_1, e_2, e_3\}$ is an orthonormal basis so that
\begin{equation*}C = 
\begin{pmatrix}
0 & a & 0 \\
0 & 0 & 0 \\
0 & 0 & 0
\end{pmatrix}
\end{equation*}
and $C$ is constant in the $T$ direction.

We now carry out some computations in this basis.
We have connection coefficients $\omega^k_{ij}$ which satisfy the following.
\begin{align*}
\nabla_T e_1 &= 0 & \nabla_T e_2 &= 0 & \nabla_T e_3 &= 0 \\
\nabla_{e_1} T &= 0 &
\nabla_{e_2} T &= - a e_1 &
\nabla_{e_3} T &= 0 \\
\nabla_{e_1} e_1 &= \omega^2_{11} e_2 + \omega^3_{11} e_3 &
\nabla_{e_2} e_1 &= -\omega^1_{22} e_2 + \omega^3_{21} e_3 + a T &
\nabla_{e_3} e_1 &= \omega^2_{31} e_2 - \omega^1_{33} e_3 \\
\nabla_{e_1} e_2 &= -\omega^2_{11} e_1 + \omega^3_{12} e_3 &
\nabla_{e_2} e_2 &= \omega^1_{22} e_1 + \omega^3_{22} e_3 &
\nabla_{e_3} e_2 &= - \omega^2_{31} e_1 - \omega^2_{33} e_3 \\
\nabla_{e_1} e_3 &= - \omega^3_{12} e_2 - \omega^3_{11} e_1 &
\nabla_{e_2} e_3 &= -\omega^3_{21} e_1 - \omega^3_{22} e_2 &
\nabla_{e_3} e_3 &= \omega^1_{33} e_1 + \omega^2_{33} e_2
\end{align*}
and hence
\begin{align*}
R(T,e_1) e_2 &= -T(\omega^2_{11}) e_1 + T(\omega^3_{12}) e_3 \\
R(T,e_1) e_3 &= -T(\omega^3_{12}) e_2  - T(\omega^3_{11}) e_1 \\
R(T, e_2) e_1 &= -(T(\omega^1_{22}) + a \omega^2_{11}) e_2 + (T(\omega^3_{21}) - a \omega^3_{11}) e_3 \\
R(T, e_2) e_3 &= (-T(\omega^3_{21}) + a \omega^3_{11}) e_1 + (-T(\omega^3_{22}) + a \omega^3_{12})) e_2 \\
R(T, e_3) e_1 &= T(\omega^2_{31}) e_2 - T(\omega^1_{33}) e_3 \\
R(T, e_3) e_2 &= - T(\omega^2_{31}) e_1 - T(\omega^2_{33}) e_3
\end{align*}
We know that all of these must be $0$ since $T \in \ker R$, and hence $\omega^2_{11}, \omega^3_{11}, \omega^3_{12}, \omega^2_{31}, \omega^1_{33}, \omega^2_{33}$ are all constant in $t$.
Moreover, $T(\omega^1_{22}) = - a \omega^2_{11}$, $T(\omega^3_{21}) = a \omega^3_{11},$ and $T(\omega^3_{22}) = a \omega^3_{12}$ show that $\omega^1_{22}, \omega^3_{22}, \omega^3_{21}$ all grow linearly in $t$.
By Poincar\'e recurrence they actually are constant, hence $\omega^2_{11} = \omega^3_{11} = \omega^3_{12} = 0$ and $\omega^1_{22}, \omega^3_{21}, \omega^3_{22}$ are constant in $t$.
In particular, all of the connection coefficients are constant along nullity geodesics.

Since $a$ is constant along nullity geodesics we get that $T(e_1(a)) = [T,e_1](a) - e_1(T(a)) = 0(a) - e_1(0) = 0$ so $e_1(a)$ is also constant along nullity geodesics.
Furthermore,
\begin{equation} T(e_2(a)) = e_2( T(a) ) - [T, e_2] (a) = a e_1(a) \end{equation}
and hence $e_2(a)$ grows linearly along nullity geodesics.
Poincar\'e recurrence again shows that $T( e_2(a) ) = a e_1(a) = 0$, so $e_1(a) = 0$.
Note that this argument shows that $e_1(f) = 0$ for any $f$ that is constant along nullity geodesics.

We also have
\begin{align*}
\chevron{R(e_2, e_3) e_3, T} &= \omega^1_{33} a \\
\chevron{R(e_3, e_2) e_2, T} &= \omega^2_{31} a \\
\chevron{R(e_2, e_1) e_1, T} &= -e_1(a) + \omega^1_{22} a = \omega^1_{22} a
\end{align*}
and since $T \in \ker R$, it follows that $\omega^1_{33} = \omega^2_{31} = \omega^1_{22} = 0$.

The second Bianchi identity gives
\begin{align*}
0
 &= \nabla_{e_1} R (e_2, e_3) e_3 + \nabla_{e_2} R(e_3, e_1) e_3 + \nabla_{e_3} R(e_1, e_3) e_2 \\
 &= \brak{e_1( e_3(\omega^3_{21}) - \omega^3_{22} \omega^3_{21}) + (\omega^2_{33})^2 \omega^3_{21}} e_1 \\
	&\quad + \brak{e_1(e_2(\omega^2_{33}) + e_3(\omega^3_{22}) - (\omega^3_{22})^2 - (\omega^2_{33})^2) + \omega^3_{22} \omega^2_{33} \omega^3_{21} - e_3(\omega^2_{33} \omega^3_{21})} e_2 \\
	&\quad + \brak{-(\omega^2_{33})^2 \omega^3_{21} - e_3(\omega^3_{22}) + e_2(\omega^2_{33}) - (\omega^3_{22})^2 - (\omega^2_{33})^2} e_3
\end{align*}
In particular, $0 = e_1(e_3(\omega^3_{21}) - \omega^3_{22} \omega^3_{21}) + (\omega^2_{33})^2 \omega^3_{21}$.
Since $T(\omega^3_{21}) = 0$, $T(e_3(\omega^3_{21})) = 0$ and so $f := e_3(\omega^3_{21}) - \omega^3_{22} \omega^3_{21}$ is constant along nullity geodesics.
By the argument above that $e_1(f) = 0$ if $T(f) = 0$, we get that $e_1( e_3(\omega^3_{21}) - \omega^3_{22} \omega^3_{21}) = 0$.
The second Bianchi identity then shows that $(\omega^2_{33})^2 \omega^3_{21} = 0$, and in particular $\omega^2_{33} \omega^3_{21} = 0$.

In summary, all of $\omega^2_{11}, \omega^3_{11}, \omega^3_{12}, \omega^1_{33}, \omega^2_{31}, \omega^3_{22}$ are zero and $\omega^2_{33} \omega^3_{21} = 0$ as well.
We use these to show that $R(e_1, \cdot) \cdot = 0$, which is a contradiction with the assumption that $V^*$ has conullity exactly 3.
Direct computation shows that, that $R(e_1, \cdot) \cdot$ is determined by:

\begin{align*}
R(e_1,e_2) e_2
	&= \brak{e_1(\omega^1_{22}) - \omega^3_{22} \omega^3_{11} + e_2(\omega^2_{11}) + \omega^3_{12} \omega^3_{21} + (\omega^2_{11})^2 + \omega^3_{12} \omega^2_{31} + (\omega^1_{22})^2 + \omega^3_{21} \omega^2_{31}} e_1 \\
	&\quad +\brak{e_1(\omega^3_{22}) + \omega^1_{22} \omega^3_{11} - e_2(\omega^3_{12}) + \omega^2_{11} \omega^3_{21} + \omega^2_{11} \omega^3_{12} + \omega^3_{12} \omega^2_{33} + \omega^1_{22} \omega^3_{22} + \omega^3_{21} \omega^2_{33}} e_3 \\
R(e_1,e_2) e_3
	&= \brak{-e_1(\omega^3_{21}) + \omega^3_{22} \omega^2_{11} + e_2(\omega^3_{11}) + \omega^3_{12} \omega^1_{22} - \omega^3_{21} \omega^1_{22} - \omega^3_{21} \omega^1_{33} - \omega^2_{11} \omega^3_{11} - \omega^2_{31} \omega^1_{33}} e_1 \\
	&\quad + \brak{e_1(\omega^3_{22}) - \omega^3_{21} \omega^2_{11} + e_2(\omega^3_{12}) - \omega^3_{11} \omega^1_{22} - \omega^1_{22} \omega^3_{22} - \omega^2_{11} \omega^3_{12} - \omega^3_{12} \omega^2_{33} - \omega^3_{21} \omega^2_{33}} e_2 \\
R(e_1, e_3) e_2
	&= \brak{-e_1(\omega^2_{31}) + \omega^2_{33} \omega^3_{11} + e_3(\omega^2_{11}) + \omega^3_{12} \omega^1_{33} + \omega^3_{12} \omega^1_{22} - \omega^3_{11} \omega^2_{11} + \omega^2_{31} \omega^1_{22} + \omega^3_{12} \omega^2_{31}} e_1 \\
	&\quad + \brak{-e_1(\omega^2_{33}) - \omega^2_{31} \omega^3_{11} - e_3(\omega^3_{12}) - \omega^2_{11} \omega^1_{33} + \omega^3_{12} \omega^3_{22} + \omega^3_{11} \omega^3_{12} + \omega^2_{31} \omega^3_{22} + \omega^1_{33} \omega^2_{33}} e_3 \\
R(e_1, e_3) e_3
	&= \brak{e_1(\omega^1_{33}) - \omega^2_{33} \omega^2_{11} + e_3(\omega^3_{11}) - \omega^2_{31} \omega^3_{12} - \omega^3_{12} \omega^3_{21} - (\omega^3_{11})^2 - \omega^2_{31} \omega^3_{21} - (\omega^1_{33})^2} e_1 \\
	&\quad + \brak{e_1(\omega^2_{33}) + \omega^1_{33} \omega^2_{11} + e_3(\omega^3_{12}) + \omega^3_{11} \omega^2_{31} - \omega^3_{12} \omega^3_{22} - \omega^3_{11} \omega^3_{12} - \omega^2_{31} \omega^3_{22} - \omega^1_{33} \omega^2_{33}} e_2
\end{align*}

Note that all the terms involving an $e_1$ derivative are zero since $e_1(f) = 0$ for all $f$ constant along nullity geodesics.
All terms involving $e_2$ or $e_3$ derivatives are zero since the connection coefficients they differentiate is zero.
All other terms involve a connection coefficient which has been shown to be zero.
Hence, $R(e_1, \cdot) \cdot$ is identically zero, which is a contradiction.


This shows that the splitting tensor $C$ is identically zero  on $V$.
So Theorem~\ref{thm:fin_volume_CN3} follows from Proposition~\ref{prop:de_rham_splitting}.
\end{proof}

Note that the hypothesis that $M$ is 4-dimensional is used only to get the vector field $T$.
In the case of $n$-manifolds that have conullity 2, $T$ was constructed in \cite{graph_manifolds} for any $n > 2$ by noting that $C_T$ is zero if self-adjoint and therefore the image of $T \mapsto C_T$ is a one-dimensional subspace of $2\times 2$ matrices.
Hence $T$ may be taken to be a vector field perpendicular to the kernel of $T \mapsto C_T$.
Such a strategy fails for conullity 3 manifolds, since the space of self-adjoint matrices is only 6 dimensional for $3 \times 3$ matrices.

\subsection{Nonnegative Curvature}

We will now prove Theorem~\ref{thm:pos_curv_CN3} using a similar strategy.
The assumption that $\Sec \geq 0$ implies that $M$ has a compact, totally geodesic soul $S$.
We start with a lemma anologous to Lemma~\ref{lemma:soul_projection}
\begin{lemma}
\label{lemma:CN3_soul_projection}
If $T \in \ker R_p$ at a point $p \in S$, then the orthogonal projections $T^S \in T_p S$ and $T^N \in T_p S^\perp$ are also in $\ker R_p$.
\end{lemma}
\begin{proof}
This is similar to the proof in the conullity 2 case.
We write $X^S$ and $X^N$ for the orthogonal projections onto $T_pS$ and $T_p S^\perp$ for any $X$.

Suppose for contradiction that $T^N$ is not in $\ker R_p$.
We may rescale $T$ to make $T^N$ unit length for simplicity.
Recall that \eqref{eqn:swap} gives that $\chevron{R(T^N, X) Y, Z} = -\chevron{R(T^S, X)Y, Z}$, and hence if we prove this result for $T^N$, it will follow for $T^S$ as well.
We choose vectors $U,V$ so that $\{T^N, U, V\}$ are orthonormal, $U,V$ are each in either $T_p S$ or $T_p S^\perp$ and they are not in $\ker R_p$.
In particular, to see that $R(T^N, \cdot) \cdot = 0$, it suffices to see that 
\begin{equation*}
    \chevron{R(T^N, X)Y, Z} = 0
\end{equation*}
for all $X,Y,Z \in \{T^N, U, V\}$.
We now proceed through the possibilities for $X,Y,Z$.

By the symmetries of $R$, we have three cases to examine:
\begin{align*}
\mbox{Case } (a): & \chevron{R(T^N, X) T^N, X} &
\mbox{Case } (b): & \chevron{R(T^N, X) X, Z} &
\mbox{Case } (c): & \chevron{R(T^N, X) T^N, Z}
\end{align*}
where $Z \perp X, T^N$.

Case (a) is just $-\Sec(T^N,X)$. Either $X \in T_p S$ or $X \in T_p S^\perp$.
If $X \in T_p S$, then $\sec(T^N, X) = 0$ since this is the curvature of a flat strip.
If $X \in T_p S^\perp$, by \eqref{eqn:swap}, $\sec(T^N, X) = \sec(T^S, X)$ which is again the curvature of a flat strip.

For case (b), we similarly first consider $X \in T_p S$.
Then $R(T^N, X) X$ is a vector in the span of $T^N$ and $X$ since the flat strips are totally geodesic, and hence the innder product with $Z$ is zero.
For the other case $X \in T_p S^\perp$, we apply \eqref{eqn:swap} again and see that
\begin{equation} \chevron{R(T^N, X^N) X^N, Z} = -\chevron{R(T^S, X^N) X^N, Z} = 0 \end{equation}
for the same reason.

Case (c) follows as in (b).
\end{proof}

This shows that there is an orthonormal basis $B = \curly{e_1, e_2, e_3, T}$ of $T_pM$, for each $p \in S$, with $T$ in $\ker R$ and $e_i$ in $\ker R^\perp$ and so that each $e_i$ and $T$ is in either $T_p S$ or in $T_p S^\perp$.
Hence we have the cases that either zero, one, two, or three of $e_1, e_2, e_3$ are in $T_p S$, and whichever of these holds at one point on $S$ must hold for all points of $S$.

\begin{proof}[Proof of Theorem~\ref{thm:pos_curv_CN3}]

First consider the cases where either none or exactly one of the $e_i$ lie in $T_p S$, which then implies that $S$ is flat.
By Lemma~\ref{lemma:flat_soul}, either $\widetilde M$ splits with a Euclidean factor, or $S$ is a point.

Next, consider the case where all three of the $e_i$ lie in $T_p S$.
If $T$ lies in $T_p S$ as well, $S$ is four dimensional, so $S = M$.
Then $M$ is compact and so $\widetilde M$ splits by Theorem~\ref{thm:fin_volume_CN3}.
If instead, $T$ lies in $T_p S^\perp$, then $S$ is a codimension 1 soul and so $M$ splits isometrically as $S \times \R$ \cite{soul_thm}.

Finally, consider the case where $e_1, e_2 \in T_p S$ but $e_3 \in T_p S^\perp$.
If $T \in T_p S$, then $S$ is codimension 1 and again $M$ splits isometrically as $M = S \times \R$.
So assume that $T \in T_p S^\perp$.
For $i, j \in \curly{1,2}$, observe that, since $S$ is totally geodesic,
\begin{equation} \chevron{ \nabla_{e_i} e_3, e_j} = - \chevron{e_3, \nabla_{e_i} e_j} = 0, \end{equation}
and also
\begin{equation} \chevron{ \nabla_{e_i} e_3, e_3} = \frac 1 2 e_i(\chevron{e_3, e_3}) = 0. \end{equation}
Since $e_i, T$ span a flat totally geodesic strip,
\begin{equation} \chevron{ \nabla_{e_i} e_3, T} = - \chevron{e_3, \nabla_{e_i} T} = 0 \end{equation}
and so we get $\nabla_{e_i} e_3 = 0$.
Similarly, $\nabla_T e_3 = 0$.
These show that $e_3$ and $T$ are parallel vector fields normal to $S$, though they may be defined only locally.
Suppose that $M$ is simply connected.
Then $e_3$ and $T$ are globally-defined parallel normal vector fields on $S$.
And hence $M$ is isometric to the space of all souls and hence splits isometrically as $M = S \times \R^2$ \cite{yim, marenich,strake}.
This completes the proof for the case that $M$ is simply connected.

For this last case with $M$ not simply connected, we then know that the universal cover $\widetilde M$ either splits isometrically or $\widetilde M$ is diffeomorphic to $\R^4$.
In the first case, we are done, so we assume that $\widetilde M \approx \R^4$.
In the current case, $M$ itself has $e_1, e_2 \in T_p S$ and $e_3, T \in T_p S^\perp$.
So, $M$ has a 2 dimensional soul $S$.
Either $S$ is flat or there is at least one point on $S$ where $\Sec(e_1, e_2) > 0$.
In the first case, Lemma~\ref{lemma:flat_soul} shows that $\widetilde M$  must split.

So suppose that $S$ has a point where $\Sec(e_1, e_2) > 0$.
Then by Gauss-Bonnet, $\widetilde S$ must be a sphere.
Since $M$ is diffeomorphic to the normal bundle $\nu(S)$, then $\widetilde M$ is diffeomorphic to the universal cover of $\nu(S)$, which is the pullback bundle $\pi^*(\nu(S))$ by $\pi: \widetilde M \rightarrow M$.
This pullback bundle is a vector bundle over $\widetilde S$, a sphere.
This contradicts the fact that $\widetilde M$ is diffeomorphic to $\R^4$.

Hence the only case when $\widetilde M$ does not split when the soul of $M$ is a point.
\end{proof}

A similar splitting result to Theorem~\ref{thm:pos_curv_CN3} was proved for arbitrary odd conullity under the assumption that the sectional curvature of all planes orthogonal to $\ker R$ are non-zero in \cite{rosenthal, rosenthal_k_nullity}.

\section{Examples}
\label{section:example}

A class of 3 dimensional examples of conullity at most 2, originating in \cite{sekigawa}, are metrics of the form
\begin{equation}
\label{eqn:sekigawa}
g = p(x,u)^2 dx^2 + (du - v \; dx)^2 + (dv + u \; dx)^2.
\end{equation}
Such manifolds have conullity exactly 2 and $\Scal = \frac {-1}{p} \frac{\partial^2 p}{\partial u^2}$.
However $\Scal > 0$, and hence $\sec \geq 0$, cannot hold for a complete manifold of this type.
Indeed, the integral curves of $\pder{}{u}$ are geodesics along which $p$ would vanish in finite time.

We now provide a modification to this which gives examples with conullity 3.
Let $M^4$ be $\R^4$ with coordinates $x,u,v,w$ and define the metric on $M$ by
\begin{equation}
g = (p(x,u,w) dx)^2 + (du - (v+w) dx)^2 + (dv + (u+w) dx)^2 + (dw - (v-u) dx)^2
\end{equation}
with $p > 0$.
Then $(M,g)$ has conullity at most 3.
To see this, define
\begin{align*}
T &:= \pder{}{v}, \\
e_1 &:= \tfrac 1 {\sqrt 2} \paren{ \pder{}{u} + \pder{}{w} }, \\
e_2 &:= \frac{1}{p(x,u,w)} \paren{\pder{}{x} + (v+w) \pder{}{u} - (u+w) \pder{}{v} + (v-u) \pder {}{x}}, \\
e_3 &:= \tfrac{1}{\sqrt 2} \paren{ \pder{}{u} - \pder{}{w} }.
\end{align*}
This gives an orthonormal basis with $T$ the nullity direction and the splitting tensor $C_T$ acting on $\{e_1,e_2, e_3\}$ is 
\begin{equation}C = 
\begin{pmatrix}
0 & \frac{\sqrt 2}{p} & 0 \\
0 & 0 & 0 \\
0 & 0 & 0
\end{pmatrix}.
\end{equation}

The $\pder{}{u}$ and $\pder{}{v}$ vector fields integrate to geodesics, as do the $e_1, e_3$ vector fields.
The hyperplanes given by $\mbox{span}\curly{T,e_1,e_3} = \mbox{span}\curly{\pder{}{u}, \pder{}{v}, \pder{}{w}}$ integrate to flat, totally geodesic submanifolds.
The scalar curvature is
\begin{equation}
\Scal = \frac {-2}{p} \paren{ \frac{\partial^2 p}{\partial w^2} + \frac{\partial^2 p}{\partial u^2} }.
\end{equation}
Moreover,  $R = 0$ if and only if $\Scal = 0$, so $\Scal > 0$ everywhere implies that the conullity is 3.

Note that this family of examples does not include a complete manifold with $\sec \geq 0$ for $g$ defined on any subset of $\R^4$.
To see this, observe that fixing any $x$ gives a totally geodesic submanifold where the induced metric is $du^2 + dv^2 + dw^2$, a flat plane.
Hence the lines in the $(u,v,w)$ planes with $x$ fixed are geodesics in $M$ and for $M$ to be complete, $g$ must be non-singular along any of these.
However, considering $p$ as a function just of $u, v, w$ on this plane, $\Delta p < 0$ everywhere, so $p$ must have a zero for some finite point and hence $g$ is singular along one of the geodesics in $M$.
Furthemore, these examples are locally irreducible since the splitting tensor $C$ does not vanish.

Finally, the result in \cite{rosenthal} gives a splitting theorem for manifolds with odd conullity under the curvature assumption that all planes orthogonal to $\ker R$ have non-zero sectional curvature.
They also prove the result for the case where $R$ is a positive or negative definite bilinear form when restricted to the space of bivectors orthogonal to $\ker R$.
We note that in our family of examples, the plane spanned by $\{e_1, e_3\}$ has $\sec = 0$ and so $(M,g)$ does not satisfy either of these curvature assumptions at any point.

\printbibliography

\end{document}